\documentclass{amsart}
\usepackage{adjustbox,lipsum}
\usepackage{amsmath}
\usepackage{amssymb}
\usepackage[shortlabels]{enumitem}
\usepackage{tabu}
\usepackage{tikz-cd}
\usetikzlibrary{positioning}
\usepackage{mathrsfs} 
\usepackage{makecell}

\usepackage{comment}
\usepackage[
        colorlinks, citecolor=darkgreen,
        backref,
        pdfauthor={Samir Siksek},
]{hyperref}

\newcommand{\Aff}{\mathbb{A}}

\newcommand{\F}{\mathbb{F}}
\newcommand{\KK}{\mathbb{K}}
\newcommand{\MM}{\mathbb{M}}

\newcommand{\Gm}{\mathbb{G}_{\mathrm{m}}}
\newcommand{\PP}{\mathbb{P}}
\newcommand{\Q}{\mathbb{Q}}

\newcommand{\Z}{\mathbb{Z}}

\newcommand{\cQ}{\mathcal{Q}}

\newcommand{\cC}{\mathscr{C}}

\newcommand{\cY}{\mathscr{Y}}

\newcommand{\fm}{\mathfrak{m}}
\newcommand{\fp}{\mathfrak{p}}

\newcommand{\fH}{\mathcal{H}}

\newcommand{\LL}{\mathbb{L}}

\newcommand{\OO}{\mathcal{O}}

\DeclareMathOperator{\Spec}{Spec}

\DeclareMathOperator{\Stab}{Stab}

\DeclareMathOperator{\genus}{genus}

\DeclareMathOperator{\Gal}{Gal}

\DeclareMathOperator{\BV}{BV}

\DeclareMathOperator{\Red}{Red}

\DeclareMathOperator{\ord}{ord}

\DeclareMathOperator{\Sym}{Sym}

\newtheorem{thm}{Theorem}

\newtheorem{lem}[thm]{Lemma}

\newtheorem{cor}[thm]{Corollary}

\theoremstyle{definition}

\theoremstyle{remark}

\definecolor{darkgreen}{rgb}{0,0.5,0}

\newcommand{\rem}{\textcolor{red}}

\begin{document}

\title{Hilbert's Irreducibility for $\mathbb{G}_m$}

\begin{abstract}
Let $K$ be a number field and $S$ a finite set of non-archimedean
places. Write $\OO_S$ for the ring of $S$-integers of $K$ and
$\OO_S^\times$ for its unit group. Let $\pi : X \rightarrow \PP^1$
	be a morphism of (irreducible) curves defined over $K$, 
	and denote by $\Red(\pi)$
	the set of $\alpha \in \PP^1(K)$
	such that the fibre $\pi^{-1}(\alpha)$
	is reducible (i.e. the Galois action
	on the fibre is intransitive).
	Hilbert's Irreducibility Theorem
	asserts that $\Red(\pi)$ is contained
	in a thin subset of $\PP^1(K)$.
	In this paper we give an explicit
	description of $\OO_S^\times \cap \Red(\pi)$.

As an application we prove the following result inspired
by a classical theorem of P\'{o}lya and Siegel. Let $p_1,\dotsc,p_s$
be rational primes. Let $f \in \Q[x]$. 
	Then the following are equivalent:
\begin{itemize}
\item There are infinitely
many tuples $(e_1,\dotsc,e_s) \in \mathbb{N}^s$ such that the polynomial $f(x)-p_1^{e_1} \cdots p_s^{e_s}$ is reducible.
\item $f=p_1^{a_1} \cdots p_s^{a_s} g^\ell$ (with $\ell$ prime) 
	or $f=-4 p_1^{a_1} \cdots p_s^{a_s} g^4$
for some $g \in \Q[x]$ and some integers
$a_1,\dotsc,a_s$.
\end{itemize}
\end{abstract}

\author{Samir Siksek}

\address{Mathematics Institute\\
    University of Warwick\\
    CV4 7AL \\
    United Kingdom}

\email{s.siksek@warwick.ac.uk}

\author{Michael Stoll}

\address{
Mathematisches Institut\\
Universit\"{a}t Bayreuth\\
95440 Bayreuth, Germany
	}
\email{michael.stoll@uni-bayreuth.de}

\date{\today}
\thanks{
Siksek is supported by the
EPSRC grant \emph{Moduli of Elliptic curves and Classical Diophantine Problems}
(EP/S031537/1). }

\makeatletter
\@namedef{subjclassname@2020}{%
  \textup{2020} Mathematics Subject Classification}
\makeatother


\maketitle

\section{Introduction}

Let $f(x) \in \Z[x]$ be a polynomial
and let
$p_1,p_2,\dotsc,p_s$ be a set of primes. 
A classical theorem of P\'{o}lya \cite{Polya}
and Siegel \cite{Siegel} 
asserts that
the equation 
\[
f(x)=p_1^{e_1} p_2^{e_2} \cdots p_s^{e_s}
\]
has only finitely many solutions with $x \in \Z$
and $e_i \ge 0$, provided $f$ has at least
two distinct roots. An effective proof was given by
Shorey and Tijdeman \cite{ShoreyTijdeman76}.
One way to state this result is to say that
if $f$ has at least two distinct roots then
there are only finitely many $s$-tuples
$(e_1,\dotsc,e_s)$ such that the polynomial
\begin{equation}\label{eqn:diff}
f(x)-p_1^{e_1} p_2^{e_2} \cdots p_s^{e_s}
\end{equation}
has a root. It is then natural to ask
if there are only finitely many tuples $(e_1,\dotsc,e_s)$
such that the polynomial \eqref{eqn:diff}
is reducible. 
This paper is concerned with 
this question, and with certain generalizations
that can be described as \lq Hilbert's
Irreducibility of $\mathbb{G}_m$\rq. The following
is a corollary to our much more general theorem (Theorem~\ref{thm:general}).
Write $\mathbb{N}=\{0,1,2,\dotsc\}$.
\begin{cor}\label{cor:stronger}
Let $f \in \Q[x]$ be a non-constant polynomial, and let $p_1,\dotsc,p_s$
be distinct primes.
The following are equivalent.
\begin{enumerate}[\upshape (i)]
\item There are $a_1,\dotsc,a_s \in \mathbb{N}$ 
	and $g(x) \in \Q[x]$ such that
\begin{equation}\label{eqn:stronger}
	p_1^{-a_1} \cdots p_s^{-a_s} \cdot f(x) \; = \; 
		-4 g(x)^4 \qquad \text{or} \qquad g(x)^\ell
\end{equation}
where, in the second case, $\ell$ is a prime.
\item There are infinitely many tuples $(e_1,\dotsc,e_s) \in \Z^s$
such that the polynomial in~\eqref{eqn:diff} is reducible.
\item There are infinitely many tuples $(e_1,\dotsc,e_s) \in \mathbb{N}^s$
such that the polynomial in~\eqref{eqn:diff} is reducible.
\end{enumerate}
\end{cor}

\bigskip

We now describe our results in greater
generality and explain the links to Hilbert's Irreducibility Theorem.
Let $K$ be a number field of degree~$d$. Let $\pi : X \rightarrow \PP^1$
be an \textbf{irreducible cover defined over $K$}; by this
we mean that $X$ is a smooth projective irreducible curve
defined over $K$, and $\pi: X \rightarrow \PP^1$
is a non-constant morphism defined over $K$. Note
that we do not insist on $X$ being absolutely irreducible.
If $X$ is absolutely irreducible then we say that
$\pi$ is an \textbf{absolutely irreducible cover}.
Write $\BV(\pi) \subset \PP^1(\overline{K})$ for the set of branch values
 of $\pi$;
this is the set of points $\alpha \in \PP^1(\overline{K})$ above which there is a point on $X$
ramified in $\pi$. For $\alpha \in \PP^1(K) \setminus \BV(\pi)$,
the fibre $\pi^{-1}(\alpha)$ consists of $d$ distinct geometric points,
where $d=\deg(\pi)$. We then
say that the fibre $\pi^{-1}(\alpha)$ is \textbf{irreducible} if and only if $\Gal(\overline{K}/K)$
acts transitively on the geometric points of the fibre $\pi^{-1}(\alpha)$, and otherwise say it is 
\textbf{reducible}.
We write
\[
\Red(\pi)=\{\alpha \in \PP^1(K) \; : \; 
\text{$\alpha \in \BV(\pi)$ or $\pi^{-1}(\alpha)$ is reducible}\}.
\]
Hilbert's Irreducibility Theorem (see, e.g.,~\cite{SerreMW})
asserts  
the existence of a thin set $B \subseteq \PP^1(K)$ containing $\Red(\pi)$.
We recall the definition of
thin sets in Section~\ref{sec:thin}, but for now 
we content ourselves with recalling the following quantitative
results; see~\cite[page 178]{SerreMW}. Let $B \subseteq \PP^1(K)$
be a thin set. Let $H(\alpha)$ denote the multiplicative height of
$\alpha \in \PP^1(K)$. Then, as $N \rightarrow \infty$,
\[
\# \{ \alpha \in B \; : \; H(\alpha) \le N\}=O(N^{3d/2} (\log{N})^\gamma),
\]
for some $\gamma<1$,
whereas
\[
\# \{ \alpha \in \PP^1(K) \; : \; H(\alpha) \le N\} \sim
c  N^{2d}
\]
for some $c>0$. Thus, in a precise sense,
relatively few $\alpha \in \PP^1(K)$ belong to $\Red(\pi)$.

We can also consider $\alpha \in \OO_K$ such that $\pi^{-1}(\alpha)$
is reducible.
If $B \subseteq \PP^1(K)$
is a thin set then, as $N \rightarrow \infty$,
\[
\# \{ \alpha \in B \cap \OO_K \; : \; H(\alpha) \le N\}
=O(N^{d/2} (\log{N})^\gamma),
\]
for some $\gamma<1$, whereas
\[
\# \{\alpha \in \OO_K \; : \; H(\alpha) \le N\} \sim c^\prime \cdot N^d,
\]
where $c^\prime>0$. Therefore again,
in a precise sense, relatively few $\alpha \in \OO_K$
belong to $\Red(\pi) \cap \OO_K$.

As motivation for this paper, we note
that $\PP^1(\OO_K)=\PP^1(K)$, as $\PP^1$ is proper,
and $\Aff^1(\OO_K)=\OO_K$, where of course $\Aff^1=\PP^1 \setminus \{\infty\}$.
Thus, Hilbert's Irreducibility Theorem asserts the irreducibility
of $\pi^{-1}(\alpha)$ for most integral points $\alpha$
on $\PP^1$, and also for most integral points $\alpha$
on $\Aff^1$.
In this paper, we give a version of Hilbert's Irreducibility
Theorem with $\alpha$  
running through the integral points of \hbox{$\Gm=\PP^1 \setminus  \{0,\infty\}$}.

We note that $\Gm(\OO_K)=\OO_K^\times$ is the unit group 
of $K$.
More generally, let $S$ be a finite set
of non-archimedean places of $K$, and write
$\OO_S$ for the ring of $S$-integers of $K$. 
Then $\Gm(\OO_S)=\OO_S^\times$ is the $S$-unit
group of $K$. We recall that $\OO_S^\times$
is finitely generated and, by Dirichlet's $S$-Unit Theorem,
has rank
$r_1+r_2+\#S -1$, where $(r_1,r_2)$ is the signature of $K$.
Thus $\OO_S^\times$ is infinite, unless $S=\emptyset$ and $K$ is either $\Q$ or an imaginary
quadratic field. 
The obvious generalization
of Hilbert's Irreducibility Theorem to $S$-integral points
on $\Gm$
is false, as illustrated by the following theorem.
\begin{thm}\label{thm:bad}
Let $K$ be a number field and let $S$ be a finite set
of non-archimedean places of $K$. Then the following hold.
\begin{enumerate}[\upshape (i)]
\item $\OO_S^\times$ is contained in a thin subset of $\PP^1(K)$.
\item There is an irreducible cover $\pi : X \rightarrow \PP^1$,
defined over $K$,
such that for every $\alpha \in \OO_S^\times$, the 
fibre $\pi^{-1}(\alpha)$ is reducible.
\end{enumerate}
\end{thm}
In spite of Theorem~\ref{thm:bad}, we give the following
result that illuminates the structure of $\Red(\pi) \cap \OO_S^\times$.
\begin{thm}\label{thm:sunitHilbert}
Let $K$ be a number field and let $S$ be a finite set
of non-archimedean places of $K$.
Let $\pi : X \rightarrow \PP^1$ be an irreducible cover
defined over $K$. Then,
there is a finite set $I \subset \OO_S^\times$, and finitely many
	pairs $(\gamma_1,r_1),\dotsc,(\gamma_k,r_k)$ with
	$\gamma_i \in \OO_S^\times$ and $r_i \ge 2$ such that
\begin{equation}\label{eqn:decomp}
	\OO_S^\times \cap \Red(\pi) \; = \; I \cup \bigcup_{i=1}^k \gamma_i \cdot
	(\OO_S^\times)^{r_i}.
\end{equation}
\end{thm}
We remark that,
given $\pi$ and $S$, 
the set of pairs $(\gamma_i,r_i)$ is effectively computable.
However, the set $I$ is not at present known to be effectively computable,
and its finiteness follows from Siegel's Theorem on integral points on
hyperbolic curves, which is ineffective.
The following is an immediate corollary to Theorem~\ref{thm:sunitHilbert}.
\begin{cor}
Under the assumptions of Theorem~\ref{thm:sunitHilbert},
suppose there is some $\alpha \in \OO_S^\times$ with $\alpha \notin \Red(\pi)$.
Then there is a finite-index subgroup $V$ of $\OO_S^\times$
such that $\pi^{-1}(\beta)$ is irreducible for all $\beta \in \alpha \cdot V$.
\end{cor}

It is natural to ask if there
is a simple criterion 
that would guarantee the finiteness of $\Red(\pi) \cap \OO_S^\times$.
The following
theorem gives such a criterion. It is convenient here to work
in greater generality and consider finitely generated subgroups of $K^\times$;
of course any such subgroup is contained in $\OO_S^\times$ for a suitable
finite set $S$ of non-archimedian places.
\begin{thm}\label{thm:general} 
Let $K$ be a number field and let $\Theta$ be a finitely generated
but infinite subgroup of $K^\times$.
Let $\pi : X \rightarrow \PP^1$ be an absolutely irreducible cover
defined over $K$. Then the following are equivalent. 
\begin{enumerate}[\upshape (I)]
\item $\Theta \cap \Red(\pi)$ is infinite. 
\item Consider $\pi$ as an element of 
	the function field $K(X)$. Then there is
		some $\gamma \in \Theta$ and some $\mu \in K(X)$
such that
		\begin{equation}\label{eqn:pimu}
			\pi \; = \; -4 \cdot \gamma \cdot \mu^4 \qquad \text{or} \qquad \gamma \cdot \mu^\ell \quad \text{(where $\ell$ is prime)}.
		\end{equation}
\end{enumerate}
\end{thm}
\begin{proof}[Proof of {\upshape ``(II)} implies {\upshape (I)''}]
The proof that (II) implies (I) is elementary and 
we give it now. 
Let $\delta \in \Theta$ and write
	\begin{equation}\label{eqn:shape}
		\beta=\gamma \delta^4 \qquad \text{or} \qquad \gamma \delta^\ell
	\end{equation}
according to whether $\pi$ satisfies the first
or second case of \eqref{eqn:pimu}. Note that
we may write $(\pi-\beta)/\gamma=\kappa_1 \cdot \kappa_2$,
where
\[
	\kappa_1=-2 \mu^2+2 \delta \mu -\delta^2, \qquad
	\kappa_2=2 \mu^2+2\delta \mu+\delta^2
\]
in the first case, and
	\[
		\kappa_1=\mu-\delta, \qquad
		\kappa_2=\mu^{\ell-1}+\delta \mu^{\ell-2}+\cdots+\delta^{\ell-1}
	\]
	in the second case. We note that in either case, $\kappa_1$, $\kappa_2 \in K(X)$
	are non-constant, and observe that
	\[
		\pi^{-1}(\beta)=\kappa_1^{-1}(0) \cup \kappa_2^{-1}(0)
	\]
where it is easy to see that the union is disjoint. 
	Thus $\pi^{-1}(\beta)$ is reducible. Since
	this is true for all $\beta$ of the shape
	\eqref{eqn:shape} with $\delta \in \Theta$,
	we conclude that $\Theta \cap \Red(\pi)$ is 
	infinite.
\end{proof}
We note that, in the above deduction of (II) from (I), we did not
use the assumption that $X$ is absolutely irreducible.
The proof that (I) implies (II), which makes use of that assumption, is given in Section~\ref{sec:general}.
We now deduce Corollary~\ref{cor:stronger} from Theorem~\ref{thm:general}.
\begin{proof}[Proof of Corollary~\ref{cor:stronger}]
We assume the hypotheses of Corollary~\ref{cor:stronger}.
Let $\pi : \PP^1 \rightarrow \PP^1$ be given by
$\pi(x)=f(x)$. This is an absolutely irreducible
cover defined over $\Q$. Let $\Theta$ be the subgroup of $\Q^\times$
generated by $\{p_1,\dotsc,p_s\}$. 

Clearly (iii) implies (ii). Let us show  
	that (i) implies (iii). 
Let $m=4$ or $\ell$ according to whether we are
	in the first or second case of \eqref{eqn:stronger}.
	Choose any $e_i \in \mathbb{N}$ such that
	$e_i \equiv a_i \pmod{m}$. Let
	\[
		h(x)=p_1^{(a_1-e_1)/m} \cdots
		p_s^{(a_s-e_s)/m} \cdot g(x).
	\]
	Then, from \eqref{eqn:stronger}, we have
	\[
		f(x)-p_1^{e_1} \cdots p_s^{e_s}=
		-p_1^{e_1} \cdots p_s^{e_s} (2 h(x)^2+2h(x)+1)(2h(x)^2-2h(x)+1),
	\]
	in the first case, and
	\[
		f(x)-p_1^{e_1} \cdots p_s^{e_s}=
		p_1^{e_1} \cdots p_s^{e_s} (h(x)-1)(h(x)^{\ell-1}+h(x)^{\ell-2}+\cdots+1),
	\]
	in the second case. This shows that (i) implies (iii).

	Finally we show that (ii) implies (i).
Suppose (ii).
Thus, by assumption there are infinitely many 
$\beta \in \Theta$ such that
the polynomial $f(x)-\beta$ is reducible, and therefore
$\pi^{-1}(\beta)$ is reducible. 
Thus $\Theta \cap \Red(\pi)$ is infinite.
By Theorem~\ref{thm:general}, there
is some $\gamma \in \Theta$ and some
$\mu(x) \in \Q(x)$ such that
\[
	f(x)=-4 \cdot \gamma \cdot \mu(x)^4 \qquad
	\text{or} \qquad f(x)=\gamma \cdot \mu(x)^\ell
	\quad \text{($\ell$ prime)}.
\]
It immediately follows
that $\mu(x) \in \Q[x]$. Now $\gamma=p_1^{e_1} \cdots p_s^{e_s}$
for some \hbox{$e_1,\dotsc,e_s \in \Z$}.
	Choosing $a_i \in \mathbb{N}$
	satisfying $a_i \equiv e_i \pmod{4}$ or $a_i \equiv e_i \pmod{\ell}$ according to whether we are in the first 
	or second case yields (i).
\end{proof}

\section{Proof of Theorem~\ref{thm:bad}}\label{sec:thin}
Let $K$ be a number field. 
We give a definition of thin subsets of $\PP^1(K)$ following
Serre \cite[chapter 9]{SerreMW}, \cite[chapter 3]{SerreTopics}.
A \textbf{type I thin subset of $\PP^1(K)$}
is simply a finite subset of $\PP^1(K)$. A \textbf{type II thin subset of
$\PP^1(K)$} has the form $\varphi(C(K))$ where $C$ is a smooth projective
absolutely irreducible curve defined over $K$, and $\varphi : C \rightarrow \PP^1$ is a surjective morphism of degree $\ge 2$ defined over $K$.
A \textbf{thin subset of $\PP^1(K)$} is a finite union of thin subsets
of types I and II. Serre gives a more general definition
of thin subsets of $\PP^n(K)$ for all $n$, but we shall not need that.
We are now ready to prove part (i) of Theorem~\ref{thm:bad}.
\begin{lem}\label{lem:badi}
Let $K$ be a number field and let $S$ be a finite set of non-archimedean
places of $K$. 
Then $\OO_S^\times$ is contained in a thin subset of $\PP^1(K)$.
\end{lem}
\begin{proof}
Recall that $\OO_S^\times$ is a finitely generated
subgroup of $K^\times$.
Let $\delta_1,\dotsc,\delta_m$ be 
coset representatives for $(\OO_S^\times)^2$ in $\OO_S$.
Let 
$\varphi_i : \PP^1 \rightarrow \PP^1$ be given by
$\varphi_i(t)= \delta_i \cdot t^2$. Then
\[
	\OO_S^\times \; = \; \bigcup_{i=1}^m \delta_i \cdot (\OO_S^\times)^2
	\; \subset \;
\bigcup_{i=1}^m \varphi_i(\PP^1(K)).
\]
The right hand-side of the inclusion is a finite union of thin sets of type II and hence a thin set.
\end{proof}

The proof of part (ii) of Theorem~\ref{thm:bad} will require
an elementary lemma from Kummer theory.
\begin{lem}\label{lem:2n}
Let $K$ be a field of characteristic $\ne 2$.
Let $A$ be a finite subset of $K^\times$,
and suppose that
	the images of $a \in A$ in the $\F_2$-vector space $K^\times/(K^\times)^2$
are linearly independent.
Then 
\begin{enumerate}[(i)]
\item $[K(\{\sqrt{a} : a \in A\}) : K]=2^{\# A}$;
\item
Let 
\[
\gamma_A=\sum_{a \in A} \sqrt{a}
\]
where $\sqrt{a}$ is any choice of square root
of $a$ in $\overline{K}$. Then 
\[
K(\gamma_A)=K(\{\sqrt{a} : a \in A\}).
\]
\end{enumerate}
\end{lem}
\begin{proof}
	Part (i) is a special case of standard theorem in Kummer Theory; see for example Theorem~VI.8.1 of 
	Lang's Algebra \cite{Lang}. For (ii) write $A=\{a_1,\dotsc,a_n\}$ and
$L=K(\sqrt{a_1},\dotsc,\sqrt{a_n})$.
Clearly
$K(\gamma_A) \subseteq L$.
The extension $L/K$ is Galois; 
write $G=\Gal(L/K)$ for the Galois group.
Let $\sigma \in G$ and suppose $\sigma(\gamma_A)=\gamma_A$.
Let $I=\{i : \sigma(\sqrt{a_i}) =-\sqrt{a_i}\}$. From $\sigma(\gamma_A)=\gamma_A$ we obtain
\[
\sum_{i \in I} \sqrt{a_i}=0. 
\]
	This contradicts part (i) if $I \ne \emptyset$.
Thus $I=\emptyset$, and therefore $\sigma=1$. It follows
that $\gamma_A$ is not contained in a proper subfield
of $L$, completing the proof.
\end{proof}

\begin{proof}[Proof of Theorem~\ref{thm:bad}]
Part (i) of Theorem~\ref{thm:bad} was established
	in Lemma~\ref{lem:badi}. We now prove (ii). 
	However, it is worth noting that Hilbert's Irreducibility
	Theorem also allows us to deduce (i) from (ii).

Recall that $K$ is a number field and $S$ is
a finite set of non-archimedean places. Therefore $\OO_S^\times$
is a finitely generated subgroup of $K^\times$.
Let $n$ be the dimension of $\OO_S^\times/(\OO_S^\times)^2$
as an $\F_2$-vector space.
Let $a_1,a_2,\dotsc,a_n$ be elements in $\OO_S^\times$
whose images in 
 $\OO_S^\times/(\OO_S^\times)^2$ form an $\F_2$-basis.
Let $\KK=K(t)$ and let 
\[
A(t)=\{t, a_1 t, a_2 t,\dotsc,a_n t\}.
\]
	Observe that the images of the elements of $A(t)$
	are $\F_2$-linearly independent in $\KK^\times/(\KK^\times)^2$.
%
Write
\[
\gamma(t)=\sqrt{t}+\sqrt{a_1 t}+\cdots+\sqrt{a_n t}.
\]
Then, by Lemma~\ref{lem:2n},
the extension $\KK(\gamma(t))/\KK$ is Galois and has degree $2^{n+1}$.
It is therefore clear that the conjugates of $\gamma(t)$ in $\KK(\gamma(t))$
are
\begin{equation}\label{eqn:signchoices}
\pm \sqrt{t} \pm \sqrt{a_1 t} \pm \cdots \pm \sqrt{a_n t},
\end{equation}
where any of the $2^{n+1}$ possible choices of signs is allowed.

Let $F \in \KK[x]$ be the minimal polynomial of $\gamma(t)$
over $\KK=K(t)$. Then $F$ is monic in $x$ of degree
	$2^{n+1}$, and moreover belongs to $K[t,x]$.
Since it is irreducible
over $K(t)$, is irreducible in $K[t,x]$. The roots of $F$
as a polynomial in $x$ are given by \eqref{eqn:signchoices}.

Now let $\alpha \in \OO_S^\times$ and consider the specialization
$F(\alpha,x)$. The roots of this in $\overline{K}$ are
\[
\pm \sqrt{\alpha} \pm \sqrt{a_1 \alpha} \pm \cdots \pm \sqrt{a_n \alpha}
\]
which are all contained in $K(\sqrt{a_1},\dotsc,\sqrt{a_n})$. 
As this is an extension of $K$ having degree $2^n$, we see that 
the degree $2^{n+1}$-polynomial $F(\alpha,x) \in K[x]$ is reducible. 
Finally, we let $X/K$ be the normalization of the irreducible
plane curve 
\[
X^\prime \; : \;	F(t,x)=0, 
\]
and we let $\pi$ be the morphism $X \rightarrow \PP^1$
which  is induced by the map $\psi : X^\prime \rightarrow \PP^1$ 
given by $\psi(t,x)=t$. It immediately follows from the 
above that $\pi$ is an irreducible cover, but that
$\pi^{-1}(\alpha)$ is reducible for all $\alpha \in \OO_S^\times$.
\end{proof}

\section{Specialisations and Hilbert's Irreducibility}\label{sec:special}
For the proof of Theorem~\ref{thm:sunitHilbert}
it is convenient to go into the details
of the proof of Hilbert's Irreducibility Theorem
for covers of $\PP^1$; this will allow us
to state a slightly more precise version
of Hilbert's Irreducibility Theorem than
is usually given. We stress that the ideas in this
section are standard, and can be found
(for example) in \cite[Section 2.4]{KoenigNeftin}.

Let $K$ be a number field and let $\pi : X \rightarrow \PP^1$
be an irreducible cover defined over $K$.
Write $\KK=K(\PP^1)$ and $\LL=K(X)$ for the function fields
of $\PP^1$ and $X$ respectively. 
By the Primitive Element Theorem,
$\LL=\KK(\theta)$, where $\theta$ is the root
of some irreducible polynomial $F(x) \in \KK[x]$ of degree 
$n=[\LL:\KK]=\deg(\pi)$.
Let $\MM$ be the splitting field of $F$;
thus $\MM=\KK(\theta_1,\dotsc,\theta_n)$ where
$\theta_1=\theta$, and the $\theta_i$ are roots of $F$ in $\MM$.
The Galois group $G=\Gal(\MM/\KK)$ is a transitive subgroup
of $\Sym(\theta_1,\dotsc,\theta_n)\cong S_n$. When
we speak of transitive subgroups of $G$, it is 
with respect to the action on $\theta_1,\dotsc,\theta_n$.

\medskip

Let $\beta$ be a 
place of $\KK$; this is simply the maximal ideal
of a valuation ring contained in $\KK$. Let $P$ be a place of $\MM$
above $\beta$. 
We define the \textbf{decomposition group} of $P/\beta$
to be
\[
D(P/\beta) = \{ \sigma \in G \; : \; \sigma(P)=P\}.
\]
Associated to $P$ is a 
\textbf{valuation ring} and a \textbf{maximal ideal}
of the valuation ring
\[
	\OO_{\MM,P}=\{f \in \MM \; : \; \ord_P(f) \ge 0\},
	\qquad
	\fm_{\MM,P}=\{f \in \OO_{\MM,P} \; : \; \ord_P(f) \ge 1\}.
\]
The \textbf{residue field} of $P$ is defined as $K(P)=\OO_{\MM,P}/\fm_{\MM,P}$.
Let $\sigma \in D(P/\beta)$. Then 
\[
	\sigma(\OO_{\MM,P})=\OO_{\MM,P}, \qquad
	\sigma(\fm_{\MM,P})=\fm_{\MM,P}.
\]
Thus $\sigma$ induces an automorphism of the residue
	field $K(P)=\OO_{\MM,P}/\fm_{\MM,P}$.
Moreover, $\sigma \in D(P/\beta) \subseteq \Gal(\MM/\KK)$ fixes pointwise
both of
\[
	\OO_{\KK,\beta}=\OO_{\MM,P} \cap \KK, 
	\qquad \fm_{\KK,\beta}=\fm_{\MM,P} \cap \KK.
\]
Thus, $\sigma$ fixes pointwise $K(\beta)=\OO_{\KK,\beta}/\fm_{\KK,\beta} \subseteq K(P)$.
Hence $\sigma$ induces an automorphism of $K(P)/K(\beta)$. 
	It turns out that $K(P)$ is a Galois extension of $K(\beta)$
	and the induced map $D(P/\beta) \rightarrow \Gal(K(P)/K(\beta))$
	is an isomorphism
(see for example \cite[Proposition 20]{SerreLocalFields}). We define
$\deg(P/\beta)=[K(P) : K(\beta)]=\#D(P/\beta)$.

Henceforth we suppose $\beta$ is a degree $1$ place of $\KK$
and thus corresponds naturally to a $K$-rational point of $\PP^1$.

Now let $H$ be a subgroup of $G$. Let $\MM^H$
be the fixed field of $H$ and let $X_H/K$ be a smooth
irreducible curve whose function field is $\MM^H$.
We write $\varphi_H : X_H \rightarrow \PP^1$ for the morphism,
defined over $K$, corresponding to the function field
extension $\MM^H/\KK$.
\begin{lem}\label{lem:degree1}
$D(P/\beta) \subseteq H$
if and only if 
the restriction of $P$ to $\MM^H$ is a degree $1$
place.
\end{lem}
\begin{proof}
Write $Q$ for the restriction of $P$ to $\MM^H$.
We note that 
\[ 
	D(P/Q) \; =\; D(P/\beta) \cap H.
\]
Then 
\[
\deg(Q/\beta)=\frac{\deg(P/\beta)}{\deg(P/Q)}=\frac{\# D(P/\beta)}{\# D(P/Q)}
=\frac{\# D(P/\beta)}{\# D(P/\beta) \cap H}.
\]
The lemma follows.
\end{proof}

\begin{thm}[Hilbert's Irreducibility Theorem]\label{thm:Hilbert}
Let $\fH$ be the set of intransitive subgroups $H$ of $G$
such that $X_H$ is absolutely irreducible. 
	Let $\beta \in \PP^1(K) \setminus \BV(\pi)$. Then
$\pi^{-1}(\beta)$ is reducible if and only if $\beta$
belongs to
\begin{equation}\label{eqn:bigcup}
\bigcup_{H \in \fH} \varphi_H(X_H(K)).
\end{equation}
Moreover, $\deg(\varphi_H) \ge 2$ for all $H \in \fH$; therefore
\eqref{eqn:bigcup} is a thin set.
\end{thm}
\begin{proof}
The action of $\Gal(\overline{K}/K)$
on the fibre $\pi^{-1}(\beta)$ induces a
permutation representation
\[
\Gal(\overline{K}/K) \; \longrightarrow \; \Sym(\pi^{-1}(\beta)) \cong S_n.
\]
The image of this permutation representation is 
	isomorphic, as a permutation group, to $\Gal(K(P)/K) \cong D(P/\beta)$, where
	$P$ is any place of $\MM$ above $\beta$ (e.g. \cite[Section 2.4]{KoenigNeftin}). 
Thus the Galois action on $\pi^{-1}(\beta)$ is intransitive if and only if
$D(P/\beta)$ is intransitive.

Suppose the Galois action on $\pi^{-1}(\beta)$ is intransitive. 
Let $P$ be any place of $\MM$
above $\beta$, and let $Q$ be its restriction to $\MM^H$, where $H=D(P/\beta)$,
which by the above discussion is intransitive.
By Lemma~\ref{lem:degree1}, the place $Q$ has degree $1$. 
Note that the residue field $K(Q)=K$ contains $\MM^H \cap \overline{K}$.
Thus 
$\MM^H \cap \overline{K}=K$ and therefore the curve $X_H/K$
is absolutely irreducible. In particular, $H \in \fH$.
Moreover, the degree $1$ place $Q$ corresponds to 
a $K$-point of $X_H$, which we also denote by $Q$.
As the place $Q$ is above $\beta$, we have $\beta=\varphi_H(Q)$.

Conversely, let $H \in \fH$, and suppose that there is some $Q \in X_H(K)$
such that $\varphi_H(Q)=\beta$. We may consider $Q$ as a degree $1$ place 
of the function field $K(X_H)=\MM^H$ above $\beta$. We extend $Q$
to a place $P$ of $\MM$. Then, by Lemma~\ref{lem:degree1} 
we see that $D(P/\beta) \subseteq H$. As $H$ is intransitive,
so is $D(P/\beta)$. Therefore Galois acts intransitively
on $\pi^{-1}(\beta)$.

It remains to show that $\deg(\varphi_H) \ge 2$ for $H \in \fH$.
We note that as $G$ is transitive, and $H \in \fH$ is intransitive,
$H \ne G$, and so $[G:H] \ge 2$. Now 
\[
	\deg(\varphi_H) \; =\; 
	[\MM^H:\KK] \; = \;
[G:H] \; \ge \; 2,
\]
	as 
the extension $\MM/\KK$ is Galois with Galois group $G$.
\end{proof}

\section{Proof of Theorem~\ref{thm:sunitHilbert}}
As before, $K$ denotes a number field, and $S$
denotes a finite number of non-archimedean
places of $K$.
For the proof of Theorem~\ref{thm:sunitHilbert} we shall
need the following famous theorem of Siegel \cite[Theorems D.8.4, D.9.1]{HindrySilverman}.
\begin{thm}[Siegel] \label{thm:Siegel}
Let $C$ be an absolutely irreducible
curve defined over $K$
and let $f \in K(C)$ be a non-constant function.
Suppose that $C$ has genus $\ge 1$, 
or that $f$ has at least three poles.
Then the set
\[
\{ P \in C(K) \; : \; f(P) \in \OO_S\}
\]
is finite.
\end{thm}
Siegel's Theorem can be more elegantly restated as a theorem
concerning integral points on hyperbolic curves \cite[Remark D.9.2.2]{HindrySilverman}, but the above version is easier to apply in our context.

\begin{lem}\label{lem:twopoints}
Let $\psi : C \rightarrow \PP^1$ be an absolutely irreducible
cover defined over $K$, and let
$r=\deg(\psi)$. Suppose $\BV(\psi) \subseteq \{0,\infty\}$.
Then, there is some $\delta \in K^\times$,
and some isomorphism $g : C \rightarrow \PP^1$ defined over $K$,
such that $\psi=\delta \cdot g^r$.
\end{lem}
\begin{proof}
By Riemann--Hurwitz applied
to $\psi$,
\[
2\genus(C) \, -\, 2 \; = \; -2r+\sum_{P \in \psi^{-1}(0)} (e_\psi(P)-1)+\sum_{P \in \psi^{-1}(\infty)} (e_\psi(P)-1);
\]
here $e_\psi(P)$ is the ramification degree of $P$ in $\psi$.
However,
\[
\sum_{P \in \psi^{-1}(0)} e_\psi(P) \; = \; \sum_{P \in \psi^{-1}(\infty)} e_\psi(P)
\; = \; r.
\]
Thus
\[
2\genus(C) \, -\, 2 \; =\; -\# \psi^{-1}(0)-\# \psi^{-1}(\infty).
\]
It follows that $\genus(C)=0$, and $\# \psi^{-1}(0)=\# \psi^{-1}(\infty)=1$.
Let $P_0$ and $P_\infty$ be the unique points on $C$
above $0$ and $\infty$ respectively. As $\psi$ is defined over $K$,
we have $P_0$, $P_\infty \in C(K)$. Thus $C$ is isomorphic to 
$\PP^1$ over $K$. Moreover, by composing this isomorphism with
a suitable automorphism of $\PP^1$, we obtain a $K$-isomorphism
$g : C \rightarrow \PP^1$ satisfying $g(P_0)=0$ and $g(P_\infty)=\infty$.
We let $h=\psi \circ g^{-1}$, giving the commutative diagram 
\[
\begin{tikzcd}
         C \arrow[d, "g"] \arrow[rd, "\psi"] & \\
         \PP^1  \arrow[r, "h" below] & \PP^1
\end{tikzcd}
\]

We note that $h(0)=0$, $h(\infty)=\infty$, that $\deg(h)=r$,
and that $\BV(h) \subseteq \{0,\infty\}$. Moreover,
$e_h(0)=e_h(\infty)=r$. Thus, $h/x^r$ has no zeros or
poles on $\PP^1$ and is therefore a constant. 
Write $\delta=h/x^r \in K^\times$.
Then $\psi=h \circ g=\delta \cdot g^r$.
\end{proof}

\begin{lem}\label{lem:infintersect}
Let $\psi : C \rightarrow \PP^1$ be an irreducible cover defined
over $K$, of degree $r \ge 1$. Suppose $\OO_S^\times \cap \psi(C(K))$ is infinite.
Let $\gamma \in \OO_S^\times \cap \psi(C(K))$. Then there is
an isomorphism $g : C \rightarrow \PP^1$
such that $\psi=\gamma \cdot g^r$. Moreover,
\[
\OO_S^\times \cap \psi(C(K)) \; = \; \gamma \cdot (\OO_S^\times)^r.
\]
\end{lem}
\begin{proof}
If $C$ is absolutely reducible then $C(K)=\emptyset$
giving a contradiction. Thus we may suppose that $C$ is absolutely 
irreducible. 

We think of $\psi$ as a non-constant function on $C$.
Since $\OO_S^\times \cap \psi(C(K))$ is infinite,
we see that
\[
\{ P \in C(K) \; : \; \psi(P) \in \OO_S^\times \}
\]
is infinite.
Let $f=\psi+\psi^{-1}$. We note that
\[
\{ P \in C(K) \; : \; f(P) \in \OO_S \} \; \supseteq
\{P \in C(K) \; : \; \psi(P) \in \OO_S^\times\}.
\]
Thus, by Siegel's Theorem,
the curve $C$ has genus $0$,
and the function $f$ has at most two poles. However, $f$
has poles at both the zeros and poles of $\psi$.
We conclude that $\psi$ has precisely one zero,
say $P_0$,
and precisely one pole, say $P_\infty$. In particular,
$e_\psi(P_0)=e_\psi(P_\infty)=r$. It follows from
the Riemann--Hurwitz Theorem that $\psi$ does not have any
branch values other that $0$ and $\infty$, so 
$\BV(\psi)=\{0,\infty\}$.

We now apply Lemma~\ref{lem:twopoints} to conclude that
$\psi=\delta \cdot g^r$ where 
	$g: C \rightarrow \PP^1$ is an isomorphism, and
	$\delta \in K^\times$.
Observe that
\[
\psi(C(K))=\delta \cdot (K^\times)^r \cup \{0,\infty\}.
\]
By assumption, the intersection 
\[
\OO_S^\times \cap \psi(C(K)) \; = \; \OO_S^\times \cap \left(
\delta \cdot (K^\times)^r \right)
\]
is infinite; 
let $\gamma$
be any element of this intersection. Then 
$\delta \cdot (K^\times)^r=\gamma \cdot (K^\times)^r$.
Thus,
$\OO_S^\times \cap \psi(C(K))=\gamma \cdot (\OO_S^\times)^r$ 
as required. Moreover, $\delta=\gamma \cdot \alpha^r$
for some $\alpha \in K^\times$. Thus, 
$\psi=\delta \cdot g^r=\gamma \cdot (\alpha g)^r$. Replacing $g$
by $\alpha g$ completes the proof.
\end{proof}

\begin{proof}[Proof of Theorem~\ref{thm:sunitHilbert}]
Let $\fH$ be as in Theorem~\ref{thm:Hilbert}.  
We divide $\fH$ into two subsets $\fH_1$, $\fH_2$,
where $\fH_1$ consists of those $H \in \fH$ 
such that the intersection 
\begin{equation}\label{eqn:intersect}
\OO_S^\times \cap \varphi_H(X_H(K))
\end{equation}
is finite, and $\fH_2$ is the remainder. 
Write $\fH_2=\{H_1,\dotsc,H_k\}$. Then, by Lemma~\ref{lem:infintersect},
for $i=1,2,\dots,k$, there is some $r_i \ge 1$ and $\gamma_i \in \OO_S^\times$
such that
\[
	\OO_S^\times \cap \varphi_{H_i}(X_{H_i}(K)) \; = \; \gamma_i \cdot (\OO_S^\times)^{r_i}.
\]
We note that $r_i=\deg(\psi_{H_i})=[\MM^{H_i}:\KK]=[G:H_i] \ge 2$,
since $H_i$ is intransitive and $G$ is transitive.
Thus
\begin{equation}\label{eqn:intersectunion}
 \OO_S^\times \cap \bigcup_{H \in \fH} \varphi_H(X_H(K)) \; = \;
I \cup \bigcup_{i=1}^k \gamma_i \cdot (\OO_S^\times)^{r_i},
\end{equation}
where $I$ is some finite set. The theorem now follows from
Theorem~\ref{thm:Hilbert}.
\end{proof}

\section{Proof of Theorem~\ref{thm:general}}\label{sec:general}
Before launching into the proof of Theorem~\ref{thm:general}
we shall need two lemmas.
The first is a theorem due to Capelli \cite{Capelli}.

\begin{lem}[Capelli]\label{lem:Lang}
Let $K$ be a field. Let $m \ge 2$ and $a \in K^\times$.
Suppose $x^m-a$ is reducible in $K[x]$. Then,
there is some $b \in K^\times$ such that,
\begin{itemize}
	\item either $4 \mid m$ and $a=-4b^4$;
	\item or $a=b^\ell$ for some prime $\ell \mid m$.
\end{itemize}
\end{lem}
\begin{proof}
For a proof see Theorem VI.9.1 of Lang's Algebra \cite{Lang}.
\end{proof}

\begin{lem}\label{lem:Galois}
Let $\KK$ be a function field and let $\MM/\KK$ be a finite Galois extension with Galois
group $G$. Let $H$ be a subgroup of $G$ and let $\LL=\MM^H$. Let $\gamma$ be a place of $\KK$,
and let $P$ be a place of $\MM$ above $\gamma$.
Write $D=D(P/\gamma)$ for the decomposition
group. Suppose there is only one place
$Q$ of $\LL$ above $\gamma$. Then
$G=D H=HD$.
\end{lem}
\begin{proof}
Let $\sigma \in G$, and consider $\sigma(P)$.
This is a place of $\MM$ above $\gamma$. Since $Q$  is
the only place of $\LL$ above $\gamma$, we conclude that
both $P$ and
$\sigma(P)$ are above $Q$. 
Thus there is some $\tau \in H
=\Gal(\MM/\LL)$ such that $\sigma(P)=\tau(P)$.
Thus $\tau^{-1} \sigma \in D$. Hence
$\sigma \in \tau D \subseteq HD$.
We conclude that $G=HD$. Finally,
	\[
		G \; =\; G^{-1} \; =\; D^{-1}H^{-1} \; = \; DH.
	\]
\end{proof}

We now turn to the proof of 
Theorem~\ref{thm:general}. We have already
shown that (II) implies (I), and it remains
to show that (I) implies (II).
Let $\pi : X \rightarrow \PP^1$ be an absolutely 
irreducible cover defined over a number field $K$,
and let $\Theta$ be a finitely generated subgroup
of $K^\times$. Suppose $\Theta \cap \Red(\pi)$
is infinite. We would like to show the existence of
		some $\gamma \in \Theta$ and some $\mu \in K(X)$
such that
\[
			\pi \; = \; -4 \cdot \gamma \cdot \mu^4 \qquad \text{or} \qquad \gamma \cdot \mu^\ell \quad \text{(where $\ell$ is prime)}.
		\]

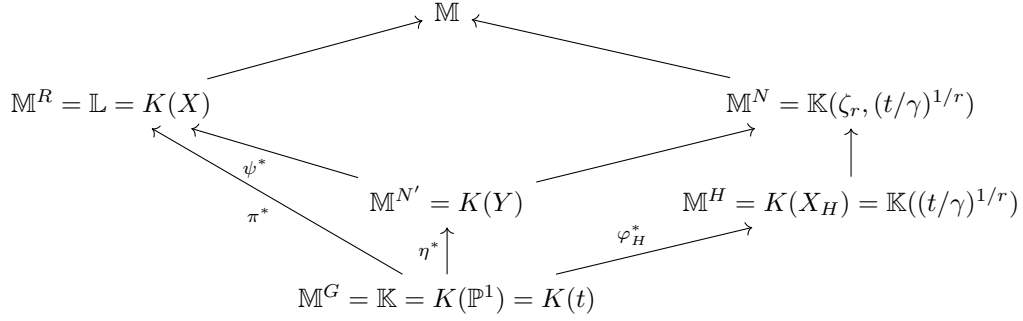
\begin{figure}
\begin{tikzcd}
	& \MM &  \\
	\MM^R=\LL=K(X) \arrow[ru] &   & \MM^{N}=\KK(\zeta_r,(t/\gamma)^{1/r}) \arrow[lu]  \\
	& \MM^{N^\prime}=K(Y) \arrow[ul, "\psi^*"] \arrow[ur]  & \MM^H=K(X_H)=\KK((t/\gamma)^{1/r}) \arrow[u]\\
	 & \MM^G=\KK=K(\PP^1)=K(t) \arrow[u, "\eta^*"] \arrow[ru, "\varphi_H^*"] 
	 \arrow[luu,"\pi^*"] & & \\
\end{tikzcd}
	\caption{
	The morphism $\pi : X \rightarrow \PP^1$
	induces an extesion $\LL/\KK$
	of function fields.
	The extension $\MM/\KK$ is the Galois
	closure of $\LL/\KK$ and its  
	Galois group is denoted by $G$. 
        In the course of the proof of
	Theorem~\ref{thm:general},
subgroups $R$, $H$, $N$, $N^\prime$
	of $G$ are constructed.
The diagram shows the relation
	between the corresponding
	fixed fields of these subgroups. }
	\label{figure:1}
\end{figure}

Write $\KK=K(t)$ for the function field of $\PP^1$,
and let $\LL=K(X)$ be the function field of $X$.
By the primitive element theorem, $\LL=\KK(\theta)$ where
$\theta$ is the root of some irreducible polynomial $F \in
\KK[x]$ of degree
$n:=[\LL:\KK]=\deg(\pi)$. Write $\MM/\KK$ for the Galois
closure of $\LL/\KK$. Then $\MM=\KK(\theta_1,\dotsc,\theta_n)$
where $\theta_1,\dotsc,\theta_n$ are distinct roots of $F$,
and $\theta_1=\theta$. Let $G=\Gal(\MM/\KK)$ which we think
of as a transitive subgroup of
$\Sym(\theta_1,\dotsc,\theta_n)$.  We consider the stabilizer
of $\theta_1$ in $G$: 
\begin{equation}\label{eqn:stab}
R:=\Stab(\theta_1)=\{\sigma \in G \; : \; \sigma(\theta_1)=\theta_1\}.
\end{equation}
We note that $\MM^R=\LL$. The reader might
find it helpful to consult Figure~\ref{figure:1}
from time to time to keep track of the 
various subfields of $\MM$ we shall construct.

We shall need to make judicious choices for 
both a degree $1$ place $\gamma$ of $\KK$,
and a place $P$ of $\MM$ above $\gamma$.
As $\Theta$ is finitely generated, there is a finite
set of non-archimedean places $S$ of $K$
such that $\Theta \subseteq \OO_S^\times$.
We now make use of the assumption that
$\Theta \cap \Red(\pi)$ is infinite.
It follows from Theorem~\ref{thm:Hilbert}
that there is an intransitive subgroup $H$
of $G$ such that $\Theta \cap \varphi_H(X_H(K))$
is infinite; here $X_H/K$ is an absolutely
irreducible curve with function field
$K(X_H)=\MM^H$, and $\varphi_H: X_H \rightarrow \PP^1$
is the morphism induced by the inclusion
$\KK \subseteq \MM^H$. Let $\gamma$ be an element
of $\Theta \cap \varphi_H(X_H(K)) \subseteq \PP^1(K)$
that is not a branch value for $\pi$;
this will be our choice of degree $1$ place on $\KK$.
By Lemma~\ref{lem:infintersect},
we have that $\varphi_H=\gamma \cdot g^r$ where $g: X_H \rightarrow \PP^1$
is an isomorphism. We will use $g$ to identify $X_H=\PP^1$. 
For now, we note that $1 \in X_H(K)$
maps to $\gamma$ in $\PP^1$. We choose $P$ to be a place of $\MM$
that is above the degree $1$ place corresponding to 
$1 \in X_H(K)$. By Lemma~\ref{lem:degree1}, we have
\begin{equation}\label{eqn:inclusion}
	D(P/\gamma) \subseteq H.
\end{equation}

We note that
\[
\MM^H=K(X_H)=\KK((t/\gamma)^{1/r}).
\]
However $\MM/\KK$ is Galois. It therefore contains the normal closure of 
$\KK((t/\gamma)^{1/r})/\KK$, which is $\KK(\zeta_r,(t/\gamma)^{1/r})$,
where $\zeta_r$ is a primitive $r$-th root of unity. Thus there is
a normal subgroup $N$ of $G$ such that
\begin{equation}\label{eqn:constant}
\MM^N=\KK(\zeta_r,(t/\gamma)^{1/r}).
\end{equation}
Since $\MM^N \supseteq \MM^H$ we have $N \subseteq H$. 

Let $N^\prime=N R$, where $R$ is given by
\eqref{eqn:stab}. As $N$ is normal, $N^\prime$ is a subgroup of $G$. 
We note that $R \subseteq N^\prime$ and $N \subseteq N^\prime$.
Thus,
\[
	\LL=\MM^R \supseteq \MM^{N^\prime}, \qquad \MM^N \supseteq \MM^{N^\prime}.
\]
In particular, $\MM^{N^\prime}$ is a 
subextension of $\LL/\KK$, and therefore 
\[
K \subseteq \MM^{N^\prime} \cap \overline{K} \subseteq \LL \cap \overline{K}
=K
\]
since $\LL=K(X)$ and $X$ is absolutely irreducible. Thus $\MM^{N^\prime}$
is the function field $K(Y)$ of some absolutely irreducible curve $Y/K$.
Moreover, $\pi \; : \; X \rightarrow \PP^1$ factors as a composition
of $K$-morphisms (say)

\begin{center}
\begin{tikzcd}
	X \arrow[r,"\psi"] 
	 \arrow[rr,out=30, in=150, "\pi"] 
	& Y \arrow[r,"\eta"] & 
	\PP^1.
\end{tikzcd}
\end{center}

Observe that the extension 
$\KK((t/\gamma)^{1/r})/\KK$ is unramified away
from the places of $\KK$ corresponding to $0$, $\infty \in \PP^1(K)$.
From \eqref{eqn:constant}, we see that
$\MM^N/\KK((t/\gamma)^{1/r})$
is a constant field extension and therefore unramified 
\cite[Theorem 3.6.3]{Stichtenoth}. As $K(Y)/\KK$
is a subextension of $\MM^N/\KK$,
the morphism $\eta : Y \rightarrow \PP^1$ is unbranched
away from $0$, $\infty$. By Lemma~\ref{lem:twopoints}
we may suppose $Y=\PP^1$ and the map $\eta : Y \rightarrow \PP^1$
is given by $\eta(z)=\delta \cdot z^m$, for some $\delta \in K^\times$,
where 
\[
m \; =\; \deg(\eta)\; =\; 
[\MM^{N^\prime}:\KK] \; =\; [G:N^\prime].
\]
We claim that $m \ge 2$; for this it is enough
to show that $N^\prime$ is a proper subgroup of $G$.
Here we make use of the identification of $G$
as a transitive subgroup of $\Sym(\theta_1,\dotsc,\theta_n)$.
Consider the orbit of $\theta_1$ under the action of $N^\prime$:
\[
N^\prime \cdot \theta_1 \; = \; N R \cdot \theta_1 \; = \; N \cdot \theta_1 \; 
\subseteq \; H \cdot \theta_1,
\]
since $R=\Stab(\theta_1)$ and $N \subseteq H$. Since $H$ is intransitive,
we conclude that $N^\prime$ is intransitive, and therefore $N^\prime$
is a proper subgroup of $G$, establishing our claim that $m \ge 2$.

We will apply Lemma~\ref{lem:Galois}
to show that the fibre $\eta^{-1}(\gamma)$ is reducible.
Suppose otherwise. Then there is precisely
one place of $K(Y)=\MM^{N^\prime}$ above $\gamma$.
Thus, by the lemma, $G=D N^\prime $ where
$D=D(P/\gamma)$ is the decomposition group. However,
$D \subseteq H$ by \eqref{eqn:inclusion}. Thus 
\[
	G\; =\; H N^\prime \; =\; H N \Stab(\theta_1)\; =\; H \Stab(\theta_1),
\]
since $N^\prime=NR$, $R=\Stab(\theta_1)$ and $H \supseteq N$.
This is a contradiction since 
\[
	H \Stab(\theta_1) \cdot \theta_1 \; =\; H \cdot \theta_1
\]
and $H$ acts intrasitively on $\{\theta_1,\dotsc,\theta_n\}$.
We conclude that $\eta^{-1}(\gamma)$ is reducible.
However, since $\eta(z)=\delta \cdot z^m$ we deduce 
that the polynomial $x^m-\gamma/\delta$ is reducible
in $K$. By Lemma~\ref{lem:Lang}, either $4\mid m$
and $\delta/\gamma=-4 b^4$ or there is
some prime $\ell \mid m$
and $\delta/\gamma = b^\ell$.
We recall that
$\pi=\eta \circ \psi=\delta \cdot \psi^m$,
where $\psi : X \rightarrow Y$. 
Thus $\pi = - 4 \gamma \mu^4$ or $\pi=\gamma \mu^{\ell}$
where $\mu=b \psi^{m/4}$ or $\mu=b \psi^{m/\ell}$ 
respectively. This completes the proof
of Theorem~\ref{thm:general}.

\bibliographystyle{abbrv}
\bibliography{bib}

\begin{thebibliography}{10}

\bibitem{Capelli}
A.~Capelli.
\newblock Sulla riduttibilit\`a{} della funzione {$x^n-A$} in un campo
  qualunque di razionalit\`a.
\newblock {\em Math. Ann.}, 54(4):602--603, 1901.

\bibitem{HindrySilverman}
M.~Hindry and J.~H. Silverman.
\newblock {\em Diophantine geometry}, volume 201 of {\em Graduate Texts in
  Mathematics}.
\newblock Springer-Verlag, New York, 2000.
\newblock An introduction.

\bibitem{KoenigNeftin}
J.~K\"onig and D.~Neftin.
\newblock Reducible fibers of polynomial maps.
\newblock {\em Int. Math. Res. Not. IMRN}, (6):5373--5402, 2024.

\bibitem{Lang}
S.~Lang.
\newblock {\em Algebra}, volume 211 of {\em Graduate Texts in Mathematics}.
\newblock Springer-Verlag, New York, third edition, 2002.

\bibitem{Polya}
G.~P\'olya.
\newblock Zur arithmetischen {U}ntersuchung der {P}olynome.
\newblock {\em Math. Z.}, 1(2-3):143--148, 1918.

\bibitem{SerreLocalFields}
J.-P. Serre.
\newblock {\em Local fields}, volume~67 of {\em Graduate Texts in Mathematics}.
\newblock Springer-Verlag, New York-Berlin, 1979.
\newblock Translated from the French by Marvin Jay Greenberg.

\bibitem{SerreMW}
J.-P. Serre.
\newblock {\em Lectures on the {M}ordell-{W}eil theorem}.
\newblock Aspects of Mathematics, E15. Friedr. Vieweg \& Sohn, Braunschweig,
  1989.
\newblock Translated from the French and edited by Martin Brown from notes by
  Michel Waldschmidt.

\bibitem{SerreTopics}
J.-P. Serre.
\newblock {\em Topics in {G}alois theory}, volume~1 of {\em Research Notes in
  Mathematics}.
\newblock A K Peters, Ltd., Wellesley, MA, second edition, 2008.
\newblock With notes by Henri Darmon.

\bibitem{ShoreyTijdeman76}
T.~N. Shorey and R.~Tijdeman.
\newblock On the greatest prime factors of polynomials at integer points.
\newblock {\em Compositio Math.}, 33(2):187--195, 1976.

\bibitem{Siegel}
C.~Siegel.
\newblock Approximation algebraischer {Z}ahlen.
\newblock {\em Math. Z.}, 10(3-4):173--213, 1921.

\bibitem{Stichtenoth}
H.~Stichtenoth.
\newblock {\em Algebraic function fields and codes}, volume 254 of {\em
  Graduate Texts in Mathematics}.
\newblock Springer-Verlag, Berlin, second edition, 2009.

\end{thebibliography}
\end{document}